\documentclass[12pt,draft]{amsart}
\usepackage{setspace,units}

\newtheorem{theorem}{Theorem}[section]
\newtheorem{lemma}[theorem]{Lemma}
\newtheorem{Corollary}[theorem]{Corollary}

\theoremstyle{definition}
\newtheorem{definition}[theorem]{Definition}

\newtheorem{remark}[theorem]{Remark}

\RequirePackage{amsmath}
\RequirePackage{amssymb}
\RequirePackage{amsthm}
 \RequirePackage{epsfig}

\begin{document}

\date{}
\title[Fourier analysis and expanding phenomena ]
{Fourier analysis and expanding phenomena in finite fields}

\author{ Derrick Hart, Liangpan Li and Chun-Yen Shen}

\address{Department of Mathematics, Rutgers University, Piscataway, NJ 08854, USA}
\email{dnhart@math.rutgers.edu}

\address{Department of Mathematics, Shanghai Jiao Tong University,
Shanghai 200240,  China $\&$ Department of Mathematics, Texas State
University, San Marcos, TX 78666, USA }
 \email{liliangpan@yahoo.com.cn}

\address{Department of Mathematics, Indiana University, Bloomington, IN 47405, USA}
\email{shenc@indiana.edu}



\date{}

\begin{abstract}
In this paper the authors study set expansion in finite fields.
Fourier analytic proofs are given for several results recently
obtained by Solymosi (\cite{So082}), Vu (\cite{Vu08}) and Vinh
(\cite{Vi09})
 using spectral graph theory.
 In addition, several generalizations of these results are given.

In the case that $A$ is a subset of a prime field $\mathbb F_p$ of
size less than $p^{1/2}$ it is shown that $|\{a^2+b:a,b \in A\}|\geq
C |A|^{147/146}$, where $|\cdot|$ denotes the cardinality of the set
and $C$ is an absolute constant.
\end{abstract}

\maketitle


\section{introduction}
Let $\mathbb F$ be a field and $E$ be a finite subset of $\mathbb
F^d$, the $d$-dimensional vector space over $\mathbb F$.  Given a
function $f:\mathbb F^d\to \mathbb F$ define
$$f(E)=\{f(x):x \in E\},$$
the image of $f$ under the subset $E$.
We shall say that $f$ is a $d$-variable expander with
expansion index $\epsilon$ if
$$|f(E)|\geq C_{\epsilon}|E|^{\nicefrac{1}{d}+\epsilon},$$
for every subset $E$
possibly under some general density or structural assumptions on
$E$.

Several classical problems in additive and geometric combinatorics
deal with showing that certain polynomials have the expander
property.  Given a finite subset $E\subset \mathbb R^d$ the Erd\H os
distance problem deals with the case of $\Delta: \mathbb R^d\times
\mathbb R^d \rightarrow \mathbb R$ where $\Delta(x,y)=\|x-y\|$. It
is conjectured that
$$|\Delta(E,E)|\gtrapprox |E|^{2/d},$$
that is $\Delta$ is a $2d$-variable expander with expansion index
$1/2d$. Taking $E$ to be a piece of the integer lattice shows that
one cannot in general do better. (Throughout the paper we will write
$X\lesssim  Y$ to mean $X \leq C Y$ where $C$ is a universal
constant, which may vary from line to line but are always universal.
It is also clear that when the quantities $X,Y$ have $f(A)$ involved
for some polynomial $f$, the implied constant may also depend on the
degree of $f$. In addition, we will write $X\lessapprox Y$ in the
case that for every $\delta>0$ there exists $C_{\delta}>0$ such that
$X\leq C_{\delta}t^{\delta}Y$ where $t$ is a large controlling
parameter.)

The fact that $\Delta$ is an expander goes back to the original 1945 by Erd\H os (\cite{E45}).
 The best results on the problem in two dimensions, due to Katz and Tardos (\cite{KT04}),
  are based on a previous breakthrough by Solymosi and Toth (\cite{SolTo01}).
  For the best known results in higher dimensions see \cite{SV04} and \cite{SV05}.

The so-called sum and product problems may also be rephrased as results about expanders.
Specifically, dealing with the fact that for a given set if one function is non-expanding
then it may imply that another function is an expander.

  Let $E=A\times\dots \times A=A^d$ where $A$ is a subset of the integers and set
$s(x_1,\dots,x_d)=x_1+\dots+x_d$ as well as
$p(x_1,\dots,x_d)=x_1\cdot \cdots \cdot x_d$.  In 1983, Erd\H os and
Szemer\'edi \cite{ESZ83} conjectured that given a set $A \subset
\mathbb Z$ that  either the size of the sumset $s(A,\dots,
A)=A+\dots +A$ or the size of of the productset $p(A,\dots,
A)=A\cdot A\cdot \dots \cdot A$ is essentially as large as possible,
that is
$$\max(|A+A+\dots+A|,|A\cdot A\cdot \dots \cdot A|)\gtrapprox |A|^{d}.$$

 By far the most studied case is that of $d=2$.  In this case  Erd\H os and Szemer\'edi gave the bound
$$\max(|A+A|,|A\cdot A|)\gtrsim |A|^{1+\epsilon},$$
 for a small but positive $\epsilon$.

Explicit bounds on $\epsilon$ where $\epsilon \geq 1/31$ was given
by Nathanson (\cite{Na97})
 and $\epsilon \geq 1/15$ by Ford (\cite{Fo98}).
  A breakthrough by Elekes (\cite{El97}) connected the problem to incidence geometry
  applying the Szem\'eredi-Trotter incidence theorem giving $\epsilon \geq 1/4$.
  This was improved by Solymosi (\cite{So05}) to $\epsilon \geq 3/14-\delta$
  where $\delta \rightarrow 0$ as $|A|\rightarrow \infty$.
  These bounds hold in the more general context of finite subsets of $\mathbb R$.
In this case the best known bound, due to Solymosi (\cite{So08}), is
given by
$$ \max (|A+A|, |A\cdot A|) \gtrapprox {|A|}^{\frac{4}{3}}.$$

With regards to the general conjecture much less is known. However,
Bourgain and Chang (\cite{BoCh04}) showed that if $A$ is a subset of
$\mathbb Z$ then for any $n\in \mathbb N$, there exists $d = d(n)$
such that for $E=A^{d}$ one has
$$\max(|s(E)|,|p(E)|) \gtrsim |A|^n.$$

The sum-product problems have been explored in the context of a
variety of rings. In this paper we will be concerned with subsets of
the finite fields $\mathbb F_q$. In this context the situation
appears to be more complicated due to the fact that one may not rely
on the topological properties of the real numbers. It is known,
however, via ground breaking work in \cite{BKT04} that if $A \subset
{\mathbb F}_p$,
 $p$ a prime, and if $|A| \leq p^{1-\epsilon}$ for some $\epsilon>0$,
 then there exists $\delta>0$ such that
$$ \max (|A+A|, |A\cdot A|) \gtrsim{|A|}^{1+\delta}.$$

This bound was given via combinatorial means and did not yield a
precise relationship between $\delta$ and $\epsilon$. In
\cite{HIS07} the first listed author along with Iosevich and
Solymosi, used Fourier analysis to develop incidence theory between
points and hyperbolas in $\mathbb F_q^2$, the $2$-dimensional vector
space over $\mathbb F_q$. This led to for the first time, a concrete
value of $\delta$, for $|A|>q^{1/2}$.

This bound on $|A|$ is natural in finite fields which are not necessarily prime fields
where subfields of size $q^{1/2}$ give the trivial bound.
Garaev (\cite{G072}) applied a method of Elekes (\cite{El97}) to give the bound
$$\max(|A+A|,|A\cdot A|)\gtrsim \min(|A|^{1/2}q^{1/2},|A|^2q^{-1/2}).$$

Solymosi (\cite{So082}) applied spectral graph theory to give a
similar bound for a general class of functions $f$ of which
polynomials of integer coefficients and degrees greater than one are
members. Let $A,B,C$ be subsets of $\mathbb F_q$. Then
$$\max(|A+B|,|f(A)+C|)\gtrsim \min(|A|^{1/2}q^{1/2},|A||B|^{1/2}|C|^{1/2}q^{-1/2}).$$
Setting $B=f(A)$ and
$C=A$ immediately gives the expander
$$|f(A)+A|\gtrsim  \min(|A|^{1/2}q^{1/2},|A|^2q^{-1/2}).$$
This result is analogous to the work done by Elekes, Nathanson and
Ruzsa (\cite{ENR00}) in the real numbers.

In \cite{Vu08} Vu classified all polynomials $f(x_1,x_2)$ for which
if $|A+A|$ is small then $|f(A,A)|$ is large.  Specifically, it was
shown that if $f$ is a ``non-degenerate" polynomial then
$$\max(|A+A|,|f(A,A)|)\gtrsim \min(|A|^{2/3}q^{1/3},|A|^{3/2}q^{-1/4}).$$

In the case of prime fields and $f(x_1,x_2)=x_1x_2$, Garaev
(\cite{G07}) used combinatorial methods to give non-trivial bounds
for all ranges of $|A|$.  In the case of $|A|\leq p^{1/2}$ then the
best result currently is due to the second listed author
(\cite{Li09}) giving
$$ \max (|A+A|, |A\cdot A|) \gtrsim {|A|}^{13/12},$$
slightly improving the result of the third listed author in \cite{She08}.

In prime fields it follows from the result of Glibichuk and Konyagin
(\cite{GK06}) that for $|A|\leq p^{1/2}$ one has that $|A\cdot
A+A|\gtrsim |A|^{7/6}$. This shows that under these constraints one
has that $f(x_1,x_2,x_3)=x_1x_2+x_3$ is
 a three-variable expander of expansion index $1/18$.
 Bourgain (\cite{Bo05}) answered a question of Widgerson
 giving examples of two-variable expanders.  Specifically, showing that
 $f(x_1,x_2)=x_1(x_1+x_2)$ and $g(x_1,x_2)=x_1(x_2+1)$ are expanders.
 However, Bourgain did not give  explicit expansion indexes.
 In the later case Garaev and the third listed author (\cite{GS09})
 showed that the expansion index could be taken to be $1/210-o(1)$.

Given the nature of an expander it seems that the image of a
sufficiently large set should expand enough to encapsulate most if
not all of a finite field.  However, some expanders exhibit a
stronger version of this property than others.  It seems expanders
fall into one of the following three types.
\begin{definition}
Let $f:\mathbb F_q^d\rightarrow \mathbb F_q$ be a given function.
\begin{itemize}
\item We say that $f$ is a \emph{strong expander} if there exists an
$\epsilon>0$ such that for all $|A|\gtrsim q^{1-\epsilon}$ one has
that $|f(A, \dots, A)|\geq q-k$ for a fixed constant $k$.
\item We say that $f$
is a \emph{moderate expander} if there exists an $\epsilon>0$ such
that for $|A|\gtrsim q^{1-\epsilon}$ one has that $|f(A, \dots,
A)|\gtrsim q$.
\item We say that $f$ is a \emph{weak expander} if there exist an
$\epsilon>0$ and a $\delta < 1$ such that for all $|A|\gtrsim
q^{1-\epsilon}$  one has that $|f(A, \dots, A)|\gtrsim
|A|^{\delta}q^{1-\delta}$.
\end{itemize}
\end{definition}

An interesting question is to determine the minimal number of
variables an expander of a certain type can have. As shown in
\cite{GS082} it is impossible for a two-variable expander to be a
strong expander. It is
unknown as to whether there are moderate expanders which are not
strong expanders for some $\epsilon$.

An example of a strong expander is the function
 $f(x_1,x_2,x_3)=x_1^2+x_1x_2+x_3$.  It was shown by Shkredov
in \cite{IS} that if
$|A|\gtrsim q^{5/6}$ then $f(A,A,A)=\mathbb Z_p$.   Shkredov
also proved that the function
 $f(x_1,x_2)=x_1(x_1+x_2)$ is a moderate expander for $|A|\gtrsim q^{3/4}$.
  A result of Garaev and the third listed author (\cite{GS09}) shows that
  $f(x_1,x_2)=x_1(x_2+1)$ is a weak expander for $|A|\gtrsim q^{2/3}$.

\section{Sum-product estimates}

Let $\mathbb G^d=G_1\times \dots \times G_d$ where $G_i\in \{\mathbb
F_q, \mathbb F_q^*\}$ and the group operation $\odot$ is inherited
from each coordinate group.  Define the Fourier transform of any
given function $f: \mathbb G^d\rightarrow \mathbb{C}$  by
$$\widehat f (\chi)=|\mathbb G^d|^{-1}\sum_{x \in \mathbb G^d}f(x)\overline{\chi}(x),$$
where $\chi=(\chi_1,\dots,\chi_d)$ and $\chi_j$ denotes the additive
or multiplicative character corresponding to $G_j$ and by the the
function $\chi(x)$ we mean $\chi_1(x_1)\cdots \chi_d(x_d)$.

We also
define the convolution of functions $f,g$ by
\[(f\ast g)(x)=\sum_{y \in \mathbb G^d}f(y)g(x\odot y^{-1}),\]
where $y^{-1}$ is the inverse of $y$ in $\mathbb G^d$. Then the
following are easy to verify:
\begin{align}
f(x)&=\sum_{\chi}\chi(x)\widehat f (\chi)\\ \widehat{f\ast
g}(\chi)&=|\mathbb G^d|\widehat f (\chi)\widehat g
(\chi),\\
\sum_{x}f(x)\overline{g(x)}&=|\mathbb{G}^d|\sum_{\chi}\widehat f
(\chi)\overline{\widehat g (\chi)}.
\end{align}

Define (\cite{TV06}) the uniformity norm (or  \emph{Fourier bias})
of $f$ by
$$\|f\|_{u}=\max_{\chi \neq \chi^0}|\widehat{f}(\chi)|,$$
where by $\chi^0$ we mean $(\chi_1^0,\dots,\chi_d^0)$ for $\chi_j^0$
the trivial character of the coordinate group. We first give a
modified version of a lemma of Solymosi (\cite{So082}).
\begin{lemma}\label{lemmabig} Suppose that $X, Y, P \subset \mathbb G^d$. Then
$$\left||\{(x,y)\in X \times Y: x\odot y \in P\}|- |X||Y||P||\mathbb G^d|^{-1}\right| \leq \|X\|_u \sqrt{|Y||P|}|\mathbb G^d|.$$
\end{lemma}

\begin{proof}
Since
\begin{align*}
|\{(x,y)\in X \times Y: x\odot y \in P\}|&=\sum_{z}(X\ast
Y)(z)P(z)\\
&=|\mathbb G^d|\sum_{\chi}\widehat{X\ast
Y}(\chi)\overline{\widehat{P}(\chi)}\\
&=|\mathbb G^d|^2\sum_{\chi}
\widehat{X}(\chi)\widehat{Y}(\chi)\overline{\widehat{P}(\chi)},\end{align*}
we have
\begin{align*}
\left||\{(x,y)\in X \times Y: x\odot y \in P\}|-|X||Y||P||\mathbb
G^d|^{-1}\right| &\leq |\mathbb G^d|^2\sum_{\chi\neq \chi^0}
|\widehat{X}(
\chi)\widehat{Y}( \chi)\overline{\widehat{P}(\chi)}|\\
&\leq |\mathbb G^d|^2 \|X\|_u\sum_{\chi\neq \chi^0}
|\widehat{Y}(\chi)||\widehat{P}(\chi)|,
\end{align*}
which in turn by Cauchy-Schwarz and Plancherel is $\leq |\mathbb
G^d|\|X\|_u \sqrt{|Y||P|}.$
\end{proof}

We say that a set $F$ is Salem with constant $C$ if
\begin{equation*}
\|F\|_{u}\leq
C\sqrt{|F|}|\mathbb G^d|^{-1}.
\end{equation*}
Let $P=X\odot Y$, where $X$ is a subset of a Salem set $\tilde X$
with constant $C$, then one has that
$$|X||Y|\leq |\{(x,y)\in \widetilde{X} \times Y: x\odot y \in P\}|
\leq|\tilde X||Y||X\odot Y||\mathbb G^d|^{-1}+C\sqrt{|\tilde X||Y||X \odot Y|}.$$
This gives the following theorem.
\begin{theorem}\label{mainbig} Suppose $X \subset \mathbb G^d$ is a subset of a Salem set $\tilde X$ with constant $C$.
Then for any $Y \subset \mathbb G^d$ one has that
$$|X\odot Y|\geq \min(|\mathbb G^d||X||\tilde X|^{-1},\;  C^{-2}|X|^2|Y||\tilde X|^{-1}).$$
\end{theorem}

\begin{remark}
This theorem can be viewed as a finite field version of the main
theorem in \cite{ENR00} by  Elekes,  Nathanson and Ruzsa, in which
the authors investigated the incidences between points and convex
curves in the real plane, and applied the incidence bound to show
$|S+T|\gtrsim\min(|S||T|,|S|^{3/2}|T|^{1/2})$ for any finite subset
$S $ of a strictly convex curve in $\mathbb{R}^2$, while $T$ is
arbitrary.
\end{remark}

\subsection{Salem Sets}

Let $\mathbb F_q$ be a finite field with characteristic $p,$ and
$Tr:\mathbb F_q\rightarrow \mathbb F_p$ be the absolute trace function. It is
well-known (\cite{LiNi97}) that the function $\widetilde{\chi}$
defined by
$$\widetilde{\chi}(c)=\exp(2\pi i Tr(c)/p), \ \ (c\in \mathbb{F}_q)$$
is a character of the additive group of $\mathbb F_q$, and every
additive character $\chi$ of $\mathbb F_q$ is of the form
$\chi(c)=\widetilde{\chi}(bc)$ for some $b\in \mathbb F_q$. Note
also the group of multiplicative characters of $\mathbb F_q$ is a
cyclic group. Denote by $N(f)$  the number of distinct roots of
$f\in \mathbb F_q[x]$ in its splitting field over $\mathbb F_q$.
Then it is easy to see that $N(f^ig^j)\leq N(f)+N(g)$ for any
$i,j\geq0$.

The classical bound due to Weil as well as its generalization for mixed
character sums may be used to show that certain sets $\tilde X$ defined by polynomials are Salem.
\begin{theorem}[Weil's Bound \cite{LiNi97,Ni02}]\label{Weil}
Let $\chi$ be a  non-trivial additive  character of $\mathbb{F}_q$
and $\psi$ be a  non-trivial multiplicative  character of
$\mathbb{F}_q$ of order $s$.
\begin{enumerate}
\item Suppose that $f\in \mathbb{F}_q[x]$ satisfies
$gcd(deg(f),q)=1$. Then  we have
$$\Big|\sum_{x \in \mathbb F_q} \chi(f(x))\Big|\leq (deg(f)-1)\sqrt{q}.$$
\item Suppose that $g\in \mathbb{F}_q[x]$ is not, up to a nonzero
multiplicative constant, an $s$-th power of a polynomial in $\mathbb
F_q[x]$. Then for any $f\in \mathbb{F}_q[x]$ we have
$$\Big|\sum_{x \in \mathbb F_q} \chi(f(x))\psi(g(x))\Big|\leq (deg(f)+d-1)\sqrt{q},$$
where $d$ is the number of distinct roots of $g$ in its splitting
field over $\mathbb{F}_q$. Particularly, taking $f$ to be some
constant function we get
$$\Big|\sum_{x \in \mathbb F_q} \psi(g(x))\Big|\leq (d-1)\sqrt{q}.$$
\end{enumerate}
\end{theorem}

\begin{Corollary}\label{Salem}
Let $p$ be the characteristic of $\mathbb{F}_q$. Suppose $f,g\in
\mathbb F_q[x]$ with $M\doteq deg(f)+deg(g)< p$ and define
$F=\{(f(x),g(x))\in \mathbb G^2\}$.
\begin{enumerate}
\item Let $G^2=\mathbb F_q\times \mathbb F_q$ and suppose $1\leq deg(f)<
deg(g)$. Or,
\item Let $G^2=\mathbb F_q\times \mathbb F_q^{\ast}$ and suppose $gcd(deg(g),q-1)=1$.
Or,
\item Let $G^2=\mathbb F_q^{\ast}\times \mathbb F_q^{\ast}$. Suppose
$f$ contains some irreducible factors that are not factors of $g$
such that the great common divisor of the powers of these factors in
the canonical factorization of $f$ is 1, and vice versa.
\end{enumerate}
Then $F$ is a Salem set with constant  $M$.
\end{Corollary}

\begin{proof}
Case (1): Suppose $(\chi_1,\chi_2)\neq(\chi^0,\chi^0)$. There exist
$b_1,b_2\in \mathbb F_q$, not all $b_i$ equal to zero, such that
\[\chi_1(c)=\widetilde{\chi}(b_1c),\ \ \chi_2(c)=\widetilde{\chi}(b_2c).\]
Thus \[\widehat{F}(\chi_1,\chi_2)=\sum_{x\in \mathbb
F_q}\chi_1(f(x))\chi_2(g(x))=\sum_{x\in \mathbb
F_q}\widetilde{\chi}(b_1f(x)+b_2g(x)).\] Since $(b_1,b_2)\neq(0,0)$
and $1\leq deg(f)<deg(g)$, we have that $b_1f(x)+b_2g(x)$ is a
polynomial of positive degree $\leq \max\{deg(f),deg(g)\}\leq M$. By
Theorem \ref{Weil}(1), $F$ is a Salem set with constant  $M$.

Case (2):  Suppose $(\chi,\psi)\neq(\chi^0,\psi^0)$ and let
$\widetilde{\psi}$ be one of its generator of the group of
multiplicative characters of $\mathbb F_q$. If $\psi=\psi^0$, then
by Theorem \ref{Weil}(2) we are done since $M< p$. Next suppose
$\psi\neq\psi^0$. Thus there exists $1\leq k\leq q-2$ such that
$\psi=(\widetilde{\psi})^k$. Thus
\[\widehat{F}(\chi,\psi)=\sum_{x\in
\mathbb F_q}\chi(f(x))\psi(g(x))=\sum_{x\in \mathbb
F_q}\chi(f(x))\widetilde{\psi}(g^k(x)).\] Since $1\leq k\leq q-2$
and $gcd(deg(g),q-1)=1$, $g^k$ could not be a $(q-1)$-th power of a
polynomial. By Theorem \ref{Weil}(2),  $F$ is a Salem set with
constant  $M$.

 Case (3): By assumption, we may write
 \[f=Q_1^{a_1}Q_2^{a_2}\cdots Q_s^{a_s}P_1^{e_1}P_2^{e_2}\cdots P_n^{e_n},\]
 \[g=R_1^{b_1}R_2^{b_2}\cdots R_t^{b_t}P_1^{f_1}P_2^{f_2}\cdots P_n^{f_n},\]
 where $Q_1,\cdots,Q_s$, $P_1,\cdots, P_n$, $R_1,\cdots, R_t$ all are distinct
 irreducible polynomials, $e_i,f_i\geq0$, $gcd(a_1,\cdots,a_s)=gcd(b_1,\cdots,b_m)=1$.
 Now suppose $(\psi_1,\psi_2)\neq(\psi^0,\psi^0)$.
 Following the notations used in  Case (2), we may write
 $\psi_1=(\widetilde{\psi})^k$, $\psi_2=(\widetilde{\psi})^j$, $0\leq k,j\leq
 q-2$, $k+j>0$. Thus
\[\widehat{F}(\psi_1,\psi_2)=\sum_{x\in
F_q}\psi_1(f(x))\psi_2(g(x))=\sum_{x\in
F_q}\widetilde{\psi}(f(x)^kg(x)^j).\] Suppose $f^kg^j$ could  be a
$(q-1)$-th power of a polynomial. Then for all $i\leq s,m\leq t$ we
have
\[(q-1)|ka_i, \ \ \ (q-1)|jb_m.\]
Thus
\[(q-1)|gcd(ka_1,\cdots,ka_s)=k, \ \ \ (q-1)|gcd(jb_1,\cdots,jb_m)=j,\]
which implies $j=k=0$, a contradiction. Therefore by Theorem
\ref{Weil}(3), $F$ is a Salem set with constant  $M$.
\end{proof}

\subsection{Sum-product estimates}

Let $p$ be the characteristic of $\mathbb{F}_q$ and $f,g\in \mathbb
F_q[x]$. Let $F=\{(f(x),g(x))\in \mathbb G^2\}$. For $A, B, C$
subsets of $\mathbb F_q$ we let $X=\{(f(x),g(x))\in \mathbb G^2: x
\in A \}$, $\tilde X=F$ and $Y =B\times C$.  Then combining Theorem
\ref{mainbig} with Corollary \ref{Salem} gives the following
generalization (at least if ones attention is restricted to
polynomials of integer coefficients) of Solymosi (\cite{So082}).

\begin{theorem}\label{bigcorollaries}
Let $p$ be the characteristic of $\mathbb{F}_q$ and $f,g\in \mathbb
F_q[x]$.
\begin{enumerate}
\item If $1\leq deg(f)<deg(g)< p$ then $$|f(A)+B||g(A)+C|\gtrsim \min(|A|q,|A|^2|B||C|q^{-1}).$$
Particularly, one has
$$|f(A)+g(A)|\gtrsim\min(|A|^{1/2}q^{1/2},|A|^2q^{-1/2}).$$
\item Suppose $gcd(deg(g),q-1)=1$ and $deg(f)\geq1, deg(f)+deg(g)< p$. Then $$|f(A)+B||g(A)C|\gtrsim  \min(|A|q,|A|^2|B||C|q^{-1}).$$
\item  Suppose
$f$ contains some irreducible factors that are not factors of $g$
such that the great common divisor of the powers of these factors in
the canonical factorization of $f$ is 1, and vice versa. Suppose
$deg(f)+deg(g)< p$. Then $$|f(A)B||g(A)C|\gtrsim
\min(|A|q,|A|^2|B||C|q^{-1}).$$ Particularly, one has
$$|f(A)g(A)|\gtrsim\min(|A|^{1/2}q^{1/2},|A|^2q^{-1/2}).$$
\end{enumerate}
\end{theorem}

\subsection{Vu's non-degenerate polynomials}

We give a generalization of Vu's result (\cite{Vu08}) using Theorem
\ref{mainbig}. Following Vu, a polynomial $P\in
\mathbb{F}_q[x_1,x_2]$ is said to be \emph{degenerate} if it is  of
the form $Q\circ L$ where $Q\in \mathbb{F}_q[x]$ and $L$ is a linear
form in $x_1,x_2$. We first recall the Schwarz-Zipple lemma
(\cite{TV06}) and the Katz theorem in \cite{Ka80}.

\begin{lemma}[Schwarz-Zipple]\label{zipple} Let $f \in \mathbb F[x_1,...,x_n]$ be a non-zero polynomial
with degree $\leq k$. Then
$$|\{x \in \mathbb{F}^n :   f(x)=0\}| \leq k|\mathbb{F}|^{n-1}.$$
\end{lemma}

\begin{theorem}[Katz]\label{salem2d}
Let $P(x_1,x_2)$ be a polynomial of degree $k$ in $\mathbb F_q^2$
which does not contain a linear factor.  Let $P^{-1}=\{(x,y)\in
\mathbb F_q^2:P(x,y)=0\}$. Then
$$\|P^{-1}\|_u\lesssim k^2q^{-3/2},$$
that is to say $P^{-1}$ is a Salem set with respect to $\mathbb
F_q^2$.
\end{theorem}

\begin{theorem}\label{VuGeneralized}
Let $P$ be a non-degenerate polynomial of degree $k$ in $\mathbb
F_q[x_1,x_2]$. Then for any $E,F\subset\mathbb F_q^2$ with $|E|\gg
k^2q$ we have
\[|P(E)|\gtrsim \min\left(\frac{|E|q}{|E+F|},\; \frac{|E||F|^{1/2}}{|E+F|^{1/2}q^{1/2}}\right).\]

\end{theorem}

\begin{proof}

 For each
$a\in\mathbb{F}_q$, let
$$P^{-1}(a)=\{(x_1,x_2)\in\mathbb{F}_q^2:\; P(x_1,x_2)=a\}.$$
By Vu's Lemma 5.1 (\cite{Vu08}), there are at least $q-(k-1)$
elements $a_i$ such that $P-a_i$ does not contain a linear factor.
We call such $a_i$ \emph{good} and form the \emph{bad} elements into
a set $\Delta$. By Lemma \ref{zipple}, for each $z\in\Delta$ one has
$|P^{-1}(z)|\leq kq$. Hence
$\sum_{z\in\Delta}|P^{-1}(z)|\leq(k-1)kq,$ and considering that
$|E|\gg k^2q$ we get
$$\Big|E\backslash\bigcup_{z\in\Delta}P^{-1}(z)\Big|\sim|E|.$$
Therefore,
$$|P(E)|\geq\frac{\displaystyle \Big|E\backslash\bigcup_{z\in\Delta}P^{-1}(z)\Big|}{M}\sim\frac{|E|}{M},$$
where
$$M=\max_{a\in\Delta^c}|E\cap P^{-1}(a)|.$$
Now choose one $a\in\Delta^c$ which achieves the above maximum and
define
$$X=E\cap P^{-1}(a), \quad \tilde X= P^{-1}(a), \quad Y=F.$$
Combining Lemma \ref{zipple}, Theorem \ref{salem2d} with the
deduction of Theorem \ref{mainbig} gives
$$\min\left(qM,\frac{M^2|F|}{q}\right)\lesssim |X+Y|\leq|E+F|.$$
Consequently,
$$M\lesssim \max\left(\frac{|E+F|}{q},\; \sqrt{q}\frac{|E+F|^{1/2}}{|F|^{1/2}}\right),$$
which in turn gives
\[|P(E)|\gtrsim \min\left(\frac{|E|q}{|E+F|},\; \frac{|E||F|^{1/2}}{|E+F|^{1/2}q^{1/2}}\right).\]
\end{proof}

\begin{remark}
Applying Theorem \ref{VuGeneralized} with $E=F=A\times A$ gives Vu's
estimate: $$\max(|A+A|,|P(A,A)|)\gtrsim
\min(|A|^{2/3}q^{1/3},|A|^{3/2}q^{-1/4}).$$

\end{remark}

\subsection{Generalized Erd\H os distance problem}
 In vector spaces over finite fields, one may define
$ \Delta(x,y)$ with
$ \|x-y\|={(x_1-y_1)}^2+\dots+{(x_d-y_d)}^2,$ and one may ask for the smallest
possible size of $\Delta(E,E)$ in terms of the size of $E$.

In this context there are additional difficulties to contend with. First,
$E$ may be the whole vector space, which would result in the rather small size for the distance set
$ |\Delta(E,E)|={|E|}^{1/d}.$
Another consideration is that if $q$ is a prime congruent to $1\ (mod\ 4)$ ,
then there exists an $i \in {\Bbb F}_q$ such that $i^2=-1$. This allows us to construct a set in ${\mathbb F}_q^2,$
$ Z=\{(t,it): t \in {\Bbb F}_q\}$ and one can easily check that
$ \Delta(Z,Z)=\{0\}.$

The first  non-trivial result on the Erd\H os distance problem in vector spaces
over finite fields was proved by Bourgain, Katz and Tao in \cite{BKT04}.
They showed that if $|E| \lesssim q^{2-\epsilon}$ for some $\epsilon>0$ and $q$ is a prime $\equiv 3~ (mod~4)$. Then
$$ |\Delta(E,E)| \gtrsim {|E|}^{\frac{1}{2}+\delta},$$ where $\delta$ is a function of $\epsilon$.

In \cite{IR07} Iosevich and Rudnev gave a distance set result for
general fields in arbitrary dimension with explicit exponents.
 They proved that if $|E| \ge 2q^{(d+1)/2}$, then $\Delta(x,y)$ is a strong expander.
 It may seem reasonable that the exponent $(d+1)/2$ is improvable. However,
 Iosevich, Koh, Rudnev and the first listed author showed in \cite{HIKR07} that
 even for the weaker conclusion that $\Delta(x,y)$ is a moderate expander then the exponent $(d+1)/2$ is sharp in odd dimensions .

In two dimensions  Chapman, Erdo\u{g}an, Iosevich, Koh and the first
listed author (\cite{CEHIK09}) showed that if $E \subset {\Bbb
F}_q^2$ satisfying $|E| \gtrsim  q^{4/3}$ then $\Delta(x,y)$ is a
moderate expander.  Iosevich and Koh (\cite{IK08}) showed that
$\Delta_n(x,y)$ with $ \|x-y\|_n={(x_1-y_1)}^n+{(x_2-y_2)}^n$ is a
strong expander for $|E| \gtrsim  q^{3/2}$. Vu (\cite{Vu08}) gave
the following result via spectral graph theory.
\begin{theorem}
Let $f \in \mathbb F_q[x_1,x_2]$ be a symmetric non-degenerate
polynomial of degree $k$ and define
$g(x_1,x_2,y_1,y_2)=f(x_1-y_1,x_2-y_2)$. Then for all $E\subset
\mathbb F_q^2$ we have
$$|g(E,E)|\gtrsim \min(k^{-1}q,\; k^{-2}|E|q^{-1/2}).$$
\end{theorem}

This effectively shows that $g$ is a four-variable moderate expander
for $|E|\gtrsim k^2q^{3/2}$. Here we give a general characterization
for which Vu's estimates hold.

\begin{theorem}\label{erdos}
Let $f \in \mathbb F_q[x_1,x_2]$ be a non-degenerate polynomial of
degree $k$ and define $g(x_1,x_2,y_1,y_2)=f(x_1-y_1,x_2-y_2)$. Then
the following two propositions are equivalent: \\ (1) $f-b$
does not contain a linear factor for any $b\in\mathbb F_q$. \\
(2) $$|g(E,F)|\gtrsim \min(k^{-1}q,\; k^{-2}\sqrt{|E||F|}q^{-1/2})$$
holds for all $E, F \subset \mathbb F_q^2$.
\end{theorem}

\begin{proof}

(1)$\Rightarrow$(2): Suppose (1) holds true. For any $b\in\mathbb
F_q$ applying Lemma \ref{lemmabig} with
$\mathbb{G}^2=\mathbb{F}_q^2, X=E, Y=-F, P=f_b$ to get
$$M_b\doteq|\{(x,y)\in E \times F :  f(x_1-y_1,x_2-y_2)=b\}|\leq \frac{|E||F||f_b|}{q^2} + \|f_b\|_u \sqrt{|E||F|}q^2,$$
where $f_b\doteq\{(z_1,z_2)\in \mathbb F_q^2 : f(z_1,z_2)=b\}.$ By
Lemma \ref{zipple} and Theorem \ref{salem2d} we get
\[M\doteq\max_{b}M_b\lesssim\max(\frac{k|E||F|}{q},\; k^2\sqrt{|E||F|q}),\]
which in turn gives
\[|g(E,F)|\geq\frac{|E||F|}{M}\gtrsim \min(k^{-1}q,\; k^{-2}\sqrt{|E||F|}q^{-1/2}).\]
(2)$\Rightarrow$(1): Suppose (2) holds true. We are trying to prove
(1) also holds true and argue it by contradiction. Suppose there
exists $\widetilde{b}\in\mathbb F_q$ such that $f-\widetilde{b}$
contains a linear factor. Thus $(f-\widetilde{b})^{-1}(0)$ must
contain a straight line, say for example $\widetilde{L}$, as a
subset. Now we choose two straight lines $E,F$ in $\mathbb{F}_q^2$
such that $E-F=\widetilde{L}$. Consequently,
$g(E,F)=\{\widetilde{b}\}$, a contradiction to (2). We are done.

\end{proof}

\subsection{Multi-fold sums and products }

\begin{theorem}\label{iteratedtheorem}
Given $A\subset \mathbb{F}_q$ and  $\oplus\in\{+,\times\}$, suppose
there exist $a,b>0$ such that for all $B\subset \mathbb{F}_q$,
\[|A\oplus B|\geq\min(a,b|B|).\]
Then for all $d\geq2$ we have
\[|d^{\oplus}A|\geq \min (a,b^{d-1}|A|),\]
where $d^{\oplus}A$ is the $d$-fold $\oplus$-set of $A$.
\end{theorem}

\begin{proof}
Define  a function $\varphi:(0,\infty)\rightarrow\mathbb(0,\infty)$
by $\varphi(x)=\min (a,\widetilde{b}x)$, where
$\widetilde{b}\doteq\max (b,1)\geq1$. It is easy to verify that
$\varphi^{(s)}(x)=\min (a,\widetilde{b}^sx)$, where
$\varphi^{(1)}=\varphi, \varphi^{(s)}=\varphi^{(s-1)}\circ\varphi$.
 By the given assumption, for all $B\subset \mathbb{F}_q$ we have
$|A\oplus B|\geq\varphi(|B|)$. Since $\varphi$ is non-decreasing, we
have
\[|d^{\oplus}A|\geq \varphi(|(d-1)^{\oplus}A|)\geq\cdots\geq \varphi^{(d-1)}(|A|)
=\min (a,\widetilde{b}^{d-1}|A|)\geq\min (a,b^{d-1}|A|).\] This
finishes the proof.

\end{proof}

Combining Theorem \ref{bigcorollaries} with the proceding theorem
naturally gives the following estimate, which improves the relevant
results in \cite{HIS07,Vi09}.

\begin{theorem}
Let $A$ be a subset of $\mathbb F_q$ and $f\in \mathbb F_q[x]$.
\begin{enumerate}
\item If $1< deg(f)<p,$ then $$|dA|\gtrsim
\min \Big(\frac{q|A|}{|f(A)+A|},\
|A|\cdot\big(\frac{|A|^{3}}{q|f(A)+A|}\big)^{d-1}\Big),$$
$$|df(A)|\gtrsim
\min\Big(\frac{q|A|}{|A+A|},\
|A|\cdot\big(\frac{|A|^{3}}{q|A+A|}\big)^{d-1}\Big);$$ If $1 \leq
deg(f) < p,$ then
$$|A^d|\gtrsim
\min \Big(\frac{q|A|}{|f(A)+A|},\
|A|\cdot\big(\frac{|A|^3}{q|f(A)+A|}\big)^{d-1}\Big),$$ and
$$|df(A)|\gtrsim
\min \Big(\frac{q|A|}{|AA|},\
|A|\cdot\big(\frac{|A|^3}{q|AA|}\big)^{d-1}\Big);$$

  \item If $f$
contains a simple root not equal to zero then
$$|A^d|\gtrsim
\min \Big(\frac{q|A|}{|f(A)A|},\
|A|\cdot\big(\frac{|A|^3}{q|f(A)A|})^{d-1}\Big),$$ and
$$|f(A)^d|\gtrsim
\min \Big(\frac{q|A|}{|AA|},\
|A|\cdot\big(\frac{|A|^3}{q|AA|}\big)^{d-1}\Big),$$
\end{enumerate}
where $dB$ and $B^d$ denote the
$d$-fold sum-set and product-set of $B$ respectively.
\end{theorem}

\section{On the expander $x+y^2$ in prime fields}

From the work of Pudl\'ak (\cite{PP06}) we know that when $|A| \leq
p^{1/2}$ with $p$ prime, one has $|\{x+y^2 : x ,y \in A\}| \gtrsim
|A|^{1+\epsilon}$ for some $\epsilon > 0.$ The arguments relied on
the finite field Szemer\'edi-Trotter incidence theorem established
by Bourgain, Katz and Tao (\cite{BKT04}). Therefore the proof did
not yield any explicit expansion index.\footnote{E. Croot
(\cite{Cr09}) gave an alternative proof using the classical version
of the Balog-Szemer\'edi-Gowers theorem which also did not yield
explicit expansion index.}

In this section we mainly give for $|A| \leq p^{1/2}$ (for $|A|\geq
p^{1/2}$ one may apply Theorem \ref{bigcorollaries}) that one has
the following explicit estimate :

\begin{theorem}\label{expanderx+y2}
Suppose $A\subset\mathbb{F}_p$ with $p$ prime and $|A|\leq p^{1/2}$.
Then one has
\[|A+A^2|\gtrsim|A|^{147/146},\]
where $A^2=\{a^2 : a \in A \}.$
\end{theorem}

Before we proceed to prove the theorem, we recall two results. The
first one is a variant of the Balog-Szemer\'edi-Gowers theorem
established by  Bourgain and Garaev (\cite{BG09}), which also played
an important role in \cite{GS09}. The second one is Garaev's type
sum-product estimate (\cite{Li09}), which was obtained by the second
listed author, improving upon the one obtained by Bourgain and
Garaev (\cite{BG09}) and the third listed author
(\cite{ShenChunYen08,She08}).

\begin{theorem}[\cite{BG09}, Lemma 2.2] \label{BourgainGaraevtheorem}
 Let $A,B$ be two sets in an abelian group $G$, and $E$ be a subset of $A\times
B$. Then there exists a subset $A'\subset A$ with
$|A'|\gtrsim|E|/|B|$ such that
\[|A\stackrel{E}{-}B|^4\gtrsim\frac{|A'- A'|\cdot|E|^5}{|A|^4\cdot|B|^3},\]
where $A\stackrel{E}{-}B=\{a-b:(a,b)\in E\}$.
\end{theorem}

\begin{theorem}[\cite{Li09}, Theorems 1.1 and 1.2, Remark 3.1]
\label{SP1213} Suppose $A\subset\mathbb{F}_p^{\ast}$ with $p$ prime
and $|A|\leq p^{12/23}$. Then for any $\oplus\in\{+,-\}$,
$\otimes\in\{\times,\div\}$, one has
\[|A\oplus A|^8\cdot |A\otimes A|^4\gtrsim|A|^{13}.\]
\end{theorem}

\begin{proof}[Proof of Theorem 3.1]
Denote
\[E=\{(x+y,\frac{1}{x-y}):x,y\in A,x\neq y, x\neq -y\}\subset ((A+A)\backslash\{0\})\times(\frac{1}{(A-A)\backslash\{0\}}).\]
Then $|E|\sim|A|^2$ and
\[((A+A)\backslash\{0\})\stackrel{E}{\div}(\frac{1}{(A-A)\backslash\{0\}})\subset A^2-A^2.\]
 Applying Theorem \ref{BourgainGaraevtheorem} with the
ambient group $\mathbb{F}_p^{\ast}$, there exists a subset $D\subset
A+A$ with
\begin{equation}\label{BSGbasicinformation}
|D|\gtrsim\frac{|A|^2}{|A-A|}
\end{equation}
so that
\begin{equation}\label{BSGpreparation}
|A^2-A^2|^4\cdot|A+A|^4\cdot|A-A|^3\gtrsim|D/D|\cdot|A|^{10}.
\end{equation}
There are two cases. \\ In the first case, suppose $|D|\leq
p^{12/23}$, then by Theorem \ref{SP1213}, we have
\[|D+D|^2|D/D|\gtrsim|D|^{3.25}.\]
Combining (\ref{BSGbasicinformation}), (\ref{BSGpreparation}) and
notice that $D\subset A+A$, we have
\[|A+A+A+A|^2\cdot|A^2-A^2|^4\cdot|A+A|^4\cdot|A-A|^3\gtrsim|D|^{3.25}\cdot|A|^{10}\gtrsim\frac{|A|^{16.5}}{|A-A|^{3.25}},\]
which gives
\[|A+A+A+A|^8\cdot|A^2-A^2|^{16}\cdot|A+A|^{16}\cdot|A-A|^{25}\gtrsim|A|^{66}.\]
Now we apply the Pl\"{u}nnecke-Ruzsa inequality as follows,
\begin{align*}
|A+A+A+A|&\leq\frac{|A+A^2|^4}{|A^2|^3}\sim\frac{|A+A^2|^4}{|A|^3},\\
|A^2-A^2|&\leq\frac{|A^2-(-A)|\cdot|(-A)-A^2|}{|A|}=\frac{|A+A^2|^2}{|A|},\\
|A+A|&\leq\frac{|A+A^2|^2}{|A^2|}\sim\frac{|A+A^2|^2}{|A|},\\
|A-A|&\leq\frac{|A-(-A^2)|\cdot|(-A^2)-A|}{|A^2|}\sim\frac{|A+A^2|^2}{|A|},
\end{align*}
to yield
\[|A+A^2|\gtrsim|A|^{147/146}.\]
We are left with the second case, $|D|\geq p^{12/23}$. Then we have
\[|A+A|\geq|D|\geq p^{12/23}\geq|A|^{24/23}.\]
But from Ruzsa's inequality, we also have
\[|A+A|\leq\frac{|A+A^2|^2}{|A^2|}\sim\frac{|A+A^2|^2}{|A|}.\]
Therefore
\[|A+A^2|^2\gtrsim|A+A|\cdot|A|\geq|A|^{47/23},\]
which yields \[|A+A^2|\gtrsim|A|^{47/46}.\]
\end{proof}

\begin{remark}
One may notice that from Theorem \ref{bigcorollaries} we have
$$|A+A^2| \gtrsim \min(|A|^{1/2}p^{1/2},|A|^2p^{-1/2}).$$
Therefore combining Theorem \ref{expanderx+y2}, one has $x+y^2$ is
an expander for all sizes of $|A|$. In addition, we notice that if
$|A|>p^{2/3},$ then
$$
|A+A^2| \gtrsim \sqrt{p\, |A|}.
$$
Let us show by adopting the Garaev-Chang example (\cite{GS09}) that
this is optimal up to the implied constant. Let $N<0.01p$ be a
positive integer, $M=[2\sqrt{Np}]$ and let X  be the set of $x$ so
that $x^2$ modulo $p$ belongs to the interval $[1,M]$. Then it is
known that $|X| \gtrsim M$. From the pigeonhole principle, there is
a number $L$ such that
$$
|X\cap \{L+1,\ldots,L+M \}| \gtrsim \frac{M^2}{2p}\sim N.
$$
Take
$$
A=X\cap \{L+1,\ldots, L+M \}.
$$
Then we have $|A|\gtrsim N$ and
$$
|A+A^2|\le 2M\lesssim \sqrt{pN}.
$$

\end{remark}

\begin{remark}\label{remarkonsecondorder}
We list some two variable polynomials of degree two:
$P_1(x,y)=x+y^2$,    $P_2(x,y)=xy+x^2$, $P_3(x,y)=xy+x$,
$P_4(x,y)=x^2+y^2$, $P_5(x,y)=x^2+xy+y^2$. Now we know $P_1$ is an
expander, and from \cite{Bo05,GS09} we also have $P_2, P_3$ are
expanders. However $P_4$ is not an expander. Since if $p$ is large
enough, we can first embed $B=\{x^2:x\in \mathbb{F}_p^{\ast} \}$
into the set of natural numbers $\mathbb{N}$, then apply
Szemer\'edi's theorem \cite{Sz75} to find a long arithmetic
progression, which in turn implies that $P_4$ is not an expander. As
to $P_5$, we currently don't know whether it is an expander or not.
\end{remark}

\begin{remark}
Suppose $\odot\in\{+,-,\times,\div\}$. Given a map $F(x,y)$, we say
$F$  is a $\odot$ operation, if
\[F(A,A)=B\odot C,\]
where $|A|\sim|B|\sim|C|$. From this notation, $P_1$ (see the
previous remark) can be thought of an additive  expander, while
$P_3$ a multiplicative expander. We will conclude this paper by
presenting an example which is a quotient expander.
\end{remark}

\begin{theorem}\label{expanderquotient}
Suppose $A\subset\mathbb{F}_p^{\ast}$ with $p$ prime and $|A|\leq
p^{1/2}$. Then one has
\[|\frac{A+1}{A}|\gtrsim|A|^{110/109}.\]
\end{theorem}

\begin{proof}
Without loss of generality we may assume $-1\not\in A$. Denote
$B=\frac{A+1}{A}$ and
\[E=\{(\frac{1}{x},\frac{y+1}{x}):x,y\in A\}\subset (1/A)\times B.\]
 Then
$|E|=|A|^2$ and  \[-A/A=(1/A)\stackrel{E}{-}B.\]  Applying Theorem
\ref{BourgainGaraevtheorem} with the ambient group $\mathbb{F}_p$,
 there exists a subset $A'\subset 1/ A$ with
$|A'|\gtrsim\frac{|A|^2}{|B|}$ so that
\[|A/A|^4\gtrsim\frac{|A'-A'|\cdot|A|^6}{|B|^3},\]
which gives
\begin{equation}\label{pp1}
|A/A|^{32}\cdot|B|^{24}\gtrsim|A'-A'|^8\cdot|A|^{48}.
\end{equation}
By Theorem \ref{SP1213},
\begin{equation}\label{pp2}
|A'-A'|^8\cdot|A'/A'|^4\gtrsim|A'|^{13}\gtrsim\frac{|A|^{26}}{|B|^{13}}.
\end{equation}
We notice $A'\subset 1/ A$, thus
\begin{equation}\label{pp3}
|A/A|^4\geq|A'/A'|^4.
\end{equation}
Multiplying (\ref{pp1}), (\ref{pp2}) and (\ref{pp3}) yields
\[|A/A|^{36}\cdot|B|^{37}\gtrsim|A|^{74}.\]
Thus  by applying the Ruzsa inequality
\[|A/A|\leq\frac{|A/(A+1)|\cdot|(A+1)/A|}{|A+1|}=\frac{|B|^2}{|A|},\]
we get
\[|B|^{109}\gtrsim|A|^{110}.\]
This finishes the proof.
\end{proof}

{\bf Acknowledgment}  LL  was  supported by the Mathematical
Tianyuan Foundation of China (10826088) and Texas Higher Education
Coordinating Board (ARP 003615-0039-2007). He also thanks Chunlei
Liu and Jian Shen  for helpful discussions.

\end{document}